\documentclass{article}

\usepackage{amsfonts,amsmath, amsthm, amssymb}
\usepackage[latin1]{inputenc}
\usepackage[english]{babel}
\usepackage{graphicx}
\usepackage{enumerate}

\newcommand{\A}{\mathcal A}

\newcommand{\K}{\mathcal K}
\newcommand{\C}{\mathcal C}

\newcommand{\G}{\mathcal G}

\newcommand{\ind}{\mathrm{i} }

\newcommand{\tw}{\rm tw}

\newcommand{\sub}{\subseteq}

\newcommand{\comment}[1]{}

\newtheorem{theorem}{Theorem}
\newtheorem{lemma}[theorem]{Lemma}
\newtheorem{corollary}[theorem]{Corollary}
\newtheorem{proposition}[theorem]{Proposition}

\theoremstyle{definition}
\newtheorem{question}{Question}
\newtheorem{conjecture}{Conjecture}

\def\?#1{\vadjust{\vbox to 0pt{\vss\vskip-8pt\leftline{%
     \llap{\hbox{\vbox{\pretolerance=-1
     \doublehyphendemerits=0\finalhyphendemerits=0
     \hsize16truemm\tolerance=10000\small
     \lineskip=0pt\lineskiplimit=0pt
     \rightskip=0pt plus16truemm\baselineskip8pt\noindent
     \hskip0pt        
     #1\endgraf}\hskip7truemm}}}\vss}}}
     
     \title{A zero-sum problem on graphs}
     \author{Daniel Wei\ss{}auer}
     \date{}
     
\begin{document}
 		\maketitle
 		
 		\begin{abstract}
%
%
 			
 			Call a graph~$G$ \emph{zero-forcing} for a finite abelian group~$\G$ if for every $\ell : V(G) \to \G$ there is a connected $A \sub V(G)$ with $\sum_{a \in A} \ell(a) = 0$. The problem we pose here is to characterise the class of zero-forcing graphs. It is shown that a connected graph is zero-forcing for the cyclic group of prime order~$p$ if and only if it has at least~$p$ vertices. 

	When $|\G|$ is not prime, however, being zero-forcing is intimately linked to the structure of the graph. We obtain partial solutions for the general case, discuss computational issues and present several questions.

 		\end{abstract}
 		
 		\begin{section}{Introduction} \label{intro}

 		Let~$\A$ be a family of subsets of some finite set~$X$ and let $(\G, +, 0)$ be a finite abelian group. We say that~$\A$ has the \emph{zero-sum property for~$\G$} if for every map ${f :  X \to \G}$ there is an $A \in \A$ with $\sum_{a \in A} f(a) = 0$. Zero-sum properties of various families~$\A$ have been studied extensively in the literature, see the surveys by Caro~\cite{CaroSurvey} and by Gao and Geroldinger~\cite{GeroldingerSurvey}. The classical Erd\"{o}s-Ginzburg-Ziv Theorem~\cite{EGZ} asserts that the family of $n$-element subsets of $\{ 1, \ldots, 2n-1 \}$ has the zero-sum property for~$\mathbb{Z}_n$, see also~\cite{AlonDubiner}. Bialostocki and Dierker proved in~\cite{BialoDiStar} that when~$X$ is the set of edges of a complete graph on $2m-1$ vertices, the family of all stars with~$m$ edges has the zero-sum property for~$\mathbb{Z}_m$ if~$m$ is even. F{\"u}redi and Kleitman showed in~\cite{FurediZeroTrees} that when~$X$ is the edge-set of a complete graph on $n+1$ vertices, the family of all spanning trees has the zero-sum property for any abelian group of order~$n$. Their result was later generalized to hypergraphs by Schrijver and Seymour~\cite{SeymourZeroTrees}, with a simplified proof.
 		
 		In the perhaps most fundamental zero-sum problem, one seeks to determine the minimum integer~$n$ such that the family of all non-empty subsets of $\{ 1, \ldots , n \}$ has the zero-sum property for~$\G$. This number is the \emph{Davenport constant} $D(\G)$ of~$\G$ and its determination is a difficult open problem (see \cite{GeroldingerSurvey}). There is a very simple folklore proof that $D(\G) \leq | \G|$: Given  $a_1, \ldots, a_{|\G|} \in \G$, consider $s_k := a_1 + \ldots + a_k$ for $1 \leq k \leq |\G|$. If one of these is zero, we are done. Otherwise, by the pigeonhole principle, two of them must be equal, say $s_k = s_m$ for $k < m$. But then $a_{k+1} + \ldots + a_m = 0$. 
 		
 		The observation that initiated the current project is that this argument does not only demonstrate the existence of \emph{any} subset summing up to zero, but even one of \emph{consecutive} indices. Put differently, the family of connected subgraphs of a path on~$|\G|$ vertices has the zero-sum property for~$\G$. We thus say that this graph is \emph{zero-forcing} for~$\G$. The question we pose and address here is which graphs other than a sufficiently long path are zero-forcing for a given finite abelian group~$\G$. 
 		
 		Precise definitions as well as some first observations and remarks will be given in Section~\ref{prelims}. We then study the special case where~$\G$ is cyclic of prime order. In that case, the above observation about the path extends in the strongest possible way.
 		
 	\begin{theorem} \label{prime trivial}
	Let~$p$ be a prime number and~$G$ a connected graph. Then~$G$ is zero-forcing for~$\mathbb{Z}_p$ if and only if~$G$ has at least~$p$ vertices.
	\end{theorem}
	
	This will be proved in Section~\ref{prime order}. When the order of~$\G$ is not prime, however, the situation appears to be more difficult and a complete characterisation of zero-forcing graphs stays out of our reach. In Section~\ref{general} we give a necessary condition for a graph to be zero-forcing, relating separation-properties of the graph to the partial order of subgroups of~$\G$, as well as some sufficient conditions for a graph to be zero-forcing. Issues related to computability and algorithms will be discussed in Section~\ref{logic}. In Section~\ref{reconstruct} we change the point of view and try to recover properties of~$\G$ from its zero-forcing graphs. Open questions and conjectures are spread throughout the text.
 		
 		\end{section}

	\begin{section}{Preliminaries} \label{prelims}
		Throughout, the letter~$G$ denotes a finite, simple and undirected graph ${G = (V,E)}$ and~$\G$ a finite abelian group with operation~$+$ and identity~0. We denote by~$\mathbb{Z}_n$ the cyclic group of integers modulo~$n$. We often do not distinguish between a set of vertices and the subgraph it induces in~$G$. In particular, we call $A \sub V(G)$ connected if and only if $G[A]$ is. For a natural number~$n$, we write $[n] := \{ 1, \ldots, n \}$.
		
		A \emph{$\G$-labeling} of~$G$ is a map $\ell : V(G) \to \G$. We call the pair $(G, \ell)$ a \emph{labeled graph}, the group~$\G$ being implicit. The labeling~$\ell$ extends to sets of vertices as $\ell(A) := \sum_{x \in A} \ell(x)$ for $A \sub V(G)$. We call $(G, \ell)$ \emph{zero-avoiding} if $\ell(A) \neq 0$ for every non-empty connected $A \sub V(G)$. We call~$G$ \emph{zero-forcing for~$\G$} if there is no $\G$-labeling~$\ell$ of~$G$ such that $(G, \ell)$ is zero-avoiding. We have the following monotonicity-property for zero-forcing graphs.
		
		\begin{proposition} \label{minorclosed}
	Let $G, H$ be two graphs. If~$H$ is a minor of~$G$ and~$H$ is zero-forcing for~$\G$, then so is~$G$.
\end{proposition}

	\begin{proof}
	Let $\ell: V(G) \to \G$ be an arbitrary $\G$-labeling of~$G$. Since~$H$ is a minor of~$G$, there is a family $(B_v \colon v \in V(H) )$ of disjoint connected $B_v \sub V(G)$ such that~$G$ contains an edge between~$B_v$ and~$B_w$ whenever $vw \in E(H)$. Define $\overline{\ell} : V(H) \to \G$ as $\overline{\ell}(v) := \ell(B_v)$. Since~$H$ is zero-forcing for~$\G$, there is a non-empty connected $A \sub V(H)$ with $\overline{\ell}(A) = 0$. But then $Q := \bigcup_{a \in A} B_a \sub V(G)$ is non-empty, connected and satisfies $\ell(Q) = 0$.
	\end{proof}
	
	By the seminal Graph Minor Theorem by Robertson and Seymour~\cite{GMT}, it follows that for every~$\G$ there is a finite set $Z(\G)$ of graphs such that a graph is zero-forcing for~$\G$ if and only if it contains some graph in $Z(\G)$ as a minor. However, it is not necessary to invoke this big machinery. We can even obtain a stronger statement from the following easy observation which initiated the current project.

	\begin{proposition} \label{paths rule}
	The path on~$n$ vertices is zero-forcing for~$\G$ iff $n \geq |\G|$.
	\end{proposition} 		
	
	\begin{proof}
	Let~$P$ be the graph with $V(P) = [n]$ and $xy \in E(P)$ iff $|x - y| = 1$.
	
	Suppose first that $n \geq |\G|$ and let $\ell : [n] \to \G$ be an arbitrary $\G$-labeling of~$P$. For $k \in [n]$ let $s_k := \ell([k])$. If one of these is equal to zero, we are done. Otherwise, by the pigeon-hole principle, two of them must be equal, say $s_k = s_m$ for $k < m$. But then $A = \{ k+1, \ldots, m \}$ is connected and $\ell(A) = 0$.
	
	Suppose now that $n < |\G|$. Let $\G = \{ g_1, g_2, \ldots \}$ be an enumeration of~$\G$. Define $\ell(j) := g_{j+1} - g_j$ for $j \in [n]$. Then for any connected $A \sub [n]$, say $A = \{k, \ldots , k+j\}$, we have $\ell(A) = g_{k+j+1} - g_k \neq 0$, so $(P, \ell)$ is zero-avoiding.
	\end{proof}
	
	Ding proved in~\cite{Ding} that for every~$k$ the class of graphs excluding the path of length~$k$ is well-quasi-ordered by the induced subgraph relation. We thus have the following corollary.
	
%
%
		\begin{corollary} \label{finite obstruction set}
	There is a finite set~$Z^{\ind}(\G)$ of graphs such that a graph is zero-forcing for~$\G$ if and only if it contains a graph in~$Z^{\ind}(\G)$ as an induced subgraph.
	\end{corollary}
	
	\begin{proof}
	Let~$Z^{\ind}(\G)$ be the set of all graphs which are zero-forcing for~$\G$, but for which no proper induced subgraph is. By Proposition~\ref{paths rule}, none of them contains a path on $|\G| +1$ vertices. It thus follows from Ding's theorem that~$Z^{\ind}(\G)$ must be finite.
	\end{proof}
	
	Taking $Z(\G)$ as the set of minor-minimal graphs in~$Z^{\ind}(\G)$, we obtain a set of graphs as guaranteed by the Graph Minor Theorem through the by far easier result of Ding. This gives, in principle, a combinatorial characterisation of zero-forcing graphs, but the proof is non-constructive. Related issues of logic and computability will be discussed in Section~\ref{logic}.
	
 		\end{section}

 		\begin{section}{Cyclic groups of prime order} \label{prime order}
 		
 		In this section we prove Theorem~\ref{prime trivial}. The condition $|G| \geq p$ is clearly necessary: Otherwise just assign the label~1 to every vertex. Our proof of the converse uses a fundamental result of Additive Combinatorics, the \emph{Cauchy-Davenport Inequality}. For non-empty $A, B \sub \mathbb{Z}_p$ we define
	\[
	A + B := \{ a + b \, \mid \, a \in A, b \in B \},  \, \, -A := \{ -a \, \mid \, a \in A \} .
	\]
	
	\begin{theorem}[Cauchy-Davenport \cite{cauchy1813,davenport1935}] 
		Let~$p$ be a prime number and $A, B \sub \mathbb{Z}_p$ non-empty. If $A + B \neq \mathbb{Z}_p$, then $|A + B| \geq |A| + |B| - 1$.
	\end{theorem}
	
	For non-empty $A_1, \ldots, A_m \sub \mathbb{Z}_p$, applying this inequality repeatedly shows that either $\sum_{j \in [m]} A_j = \mathbb{Z}_p$ or 
	\[
	\left| \sum_{j \in [m]} A_j \right| \geq 1 + \sum_{j \in [m]} (|A_j| - 1) .
	\]
	
	The statement of Theorem~\ref{prime trivial} is not prone to induction, so instead we prove something seemingly stronger. For a labeled graph $(G, \ell)$ and $v \in V(G)$ define $Q(v,G) := \{ \ell(A) \colon v \in A \sub G \text{ connected} \}$, the set of all values that can be ``generated'' at~$v$.

	\begin{lemma} \label{tree prime}
		Let~$T$ be a tree, $\ell : V(T) \to \mathbb{Z}_p$. If $(T, \ell)$ is zero-avoiding, then $|Q(t,T)| \geq |T|$ for every $t \in T$.
	\end{lemma}	

	This certainly implies Theorem~\ref{prime trivial}: Just take a spanning tree of~$G$ and use monotonicity.
	
	\begin{proof}[Proof of Lemma~\ref{tree prime}]
		We prove this by induction on~$|T|$. If~$T$ consists of a single vertex, this is clearly true. For the inductive step, let $(T, \ell)$ be zero-avoiding and $t \in T$. Let $T_1, \ldots, T_d$ be the components of $T - t$ and $t_1, \ldots, t_d$ the neighbors of~$t$ in each of these, respectively. 
		
		The connected $t \in A \sub T$ are precisely the sets $A = \{ t \} \cup \bigcup_{j \in [d]} A_j$ with each~$A_j$ either being empty or $t_j \in A_j \sub T_j$ connected. This implies that
		\[
		Q(t,T) = \{ \ell(t) \} + \sum_{j \in [d]} (Q(t_j,T_j) \cup \{ 0 \} ) .
		\]
		
		The restrictions $(T_j, \ell|_{T_j})$ are zero-avoiding, so by inductive hypothesis we have $|Q(t_j,T_j)| \geq |T_j|$ for each $j \in [d]$. Since $0 \notin Q(t_j, T_j)$, the iterated Cauchy-Davenport Inequality yields
		\[
		|Q(t,T)| \geq 1 + \sum_{j \in [d]} |Q(t_j, T_j)| \geq 1 + \sum_{j \in [d]} |T_j| = |T|.
		\]
	\end{proof}

	Let us briefly explain how the Cauchy Davenport Inequality can in turn be deduced from Theorem~\ref{prime trivial}. Let~$p$ prime and $A, B \sub \mathbb{Z}_p$ non-empty such that $A + B \neq \mathbb{Z}_p$, thus also $C:= \mathbb{Z}_p \setminus -(A+B)$ is non-empty. 
	
	Enumerate the three sets, say $A = \{a_1, \ldots, a_{k+1}\}, B = \{ b_1, \ldots, b_{m+1}\}$ and $C = \{ c_1, \ldots, c_{n+1} \}$. Take three disjoint paths $u_1 \ldots u_k$, $v_1\ldots v_m$ and $w_1 \ldots w_n$, an extra vertex~$x$ and join~$x$ to $u_1, v_1$ and~$w_1$. Assign the labels $\ell(x) = a_1 + b_1 + c_1$, $\ell(u_i) = a_{i+1} - a_i$, $\ell(v_i) = b_{i+1} - b_i$ and $\ell(w_i) = c_{i+1} - c_i$ for all indices in the appropriate ranges, respectively. By definition, the resulting labeled graph is zero-avoiding, so by Theorem~\ref{prime trivial} we must have $k + m + n + 1 < p$ or, in other words, $|A + B|  > k+m$.
 		\end{section}

 		\begin{section}{The general case} \label{general}
 		
 		Let us now consider the case where~$|\G|$ is not prime. A statement analogous to Theorem~\ref{prime trivial} is then no longer true (an example will be given below) and the property of being zero-forcing is tied to the structural richness of the graph.
 		
		In light of Proposition~\ref{paths rule} it is clear that even sparse graphs such as paths can be zero-forcing once they have a certain order. Hence the property of being zero-forcing does not provide a strong measure of ``width" on its own and we should take the order of the graph into account as well. The general kind of question we are concerned with here is of the following type.
 		
 		\begin{question} \label{Q sparse min order}
 			Given some class~$\C$ of graphs, what is the minimum order of a graph in~$\C$ that is zero-forcing for~$\G$?
 		\end{question}

 		\begin{subsection}{Necessary conditions}

	Given Theorem~\ref{prime trivial} and Proposition~\ref{paths rule}, it might seem that the property of being zero-forcing depended solely on the order of the graph and not on its structure. To see that this is indeed not the case, consider the group~$\mathbb{Z}_4$ and a star whose center is labeled~1 and whose leaves receive the label~2: No matter how large we choose the star, it is zero-avoiding. The idea behind this example is captured by the following.
	
	\begin{lemma} \label{separator approach}
		Let~$\G'$ be a subgroup of~$\G$ and $X \sub V(G)$ with $|X| < D(\G / \G')$. If~$G$ is zero-forcing for~$\G$, then some component of $G - X$ is zero-forcing for~$\G'$.
	\end{lemma}
	
	\begin{proof}
	Suppose this was not the case, so every component~$C$ of $G - X$ has a $\G$-labeling~$\ell_C$ such that $(C, \ell_C)$ is zero-avoiding and~$\ell_C$ takes values only in~$\G'$. Since $|X| < D( \G / \G')$, there is a map $\ell_X : X \to \G$ such that for every non-empty $Y \sub X$ we have $\ell_X(Y) \notin \G'$.
	
	Combine these labelings to a $\G$-labeling~$\ell$ of~$G$. Every connected subgraph~$A$ of~$G$ is either contained in some component~$C$ of $G - X$, so that $\ell(A) = \ell_C(A) \neq 0$, or meets~$X$ so that $\ell(A) \notin \G'$ is non-zero. Thus $(G, \ell)$ is zero-avoiding.
	\end{proof}
 		
	This lemma provides a simple recursive method for constructing zero-avoiding labeled graphs. It does not suffice to capture all there is to it, however, as the example of a zero-avoiding $\mathbb{Z}_3 \times \mathbb{Z}_3$-labeled graph in Figure~\ref{counterex 3x3} shows: one cannot deduce from Lemma~\ref{separator approach} that the graph is not zero-forcing.
	
	\begin{figure}[h]
		\begin{center}
			\includegraphics[width=4cm]{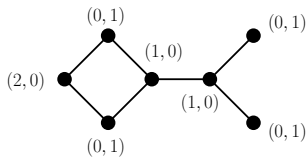}
		\end{center}
		\caption{A zero-avoiding $\mathbb{Z}_3 \times \mathbb{Z}_3$-labeled graph}
		 \label{counterex 3x3}
	\end{figure}

	We illustrate the use of Lemma~\ref{separator approach} by proving the following.
	
 	\begin{proposition} \label{small trees}
		Let~$T$ be a tree on less than $|\G|$ vertices. Then~$T$ is not zero-forcing for~$\G$.
	\end{proposition}	
	
		Combined with Proposition~\ref{paths rule}, this shows that the minimum order of a zero-forcing tree is precisely the order of~$\G$, thus answering Question~\ref{Q sparse min order} for the class of trees. Besides Lemma~\ref{separator approach}, the proof is based on the following well-known separation property of trees. We include a proof for completeness.
	
	\begin{lemma} \label{tree sep}
		Let $m, n$ be positive integers and~$T$ a tree of order less than~$mn$. Then there exists $X \sub V(T)$, $|X| < m$, such that every component of $T - X$ has order less than~$n$.
	\end{lemma}
	
	\begin{proof}
		By induction on~$m$, the case $m=1$ being trivial. Choose an arbitrary root~$r$ for~$T$, thus inducing a partial order on~$V(T)$. If $|T| < n$, we may take $X  = \emptyset$. Otherwise choose $x \in V(T)$ maximal in the tree-order such that the subtree rooted at~$x$ has order at least~$n$. Every component of $T - x$ not containing the root has order less than~$n$. Hence if $x = r$, we simply take $X = \{ x \}$. Otherwise, let~$S$ be the component of $T - x$ containing~$r$. Since $|S| \leq |T| - n < (m-1)n$, we can apply the inductive hypothesis to $(S, m-1, n)$ and find $Y \sub V(S)$, $|Y| < m-1$, such that every component of $S - Y$ has order less than~$n$. Then $Y \cup \{ x \}$ is as desired.
	\end{proof}

	\begin{proof}[Proof of Proposition~\ref{small trees}]
	By induction on~$|\G|$. Since the statement is trivial when~$\G$ is cyclic, we may assume that~$\G$ has a non-trivial proper subgroup. Let~$\G'$ be a maximal proper subgroup of~$\G$. Then $\G / \G' \cong \mathbb{Z}_p$ for some prime number~$p$. By Lemma~\ref{tree sep} applied to $(T, p, |\G'|)$, there is some $X \sub V(T)$ with $|X| < p$ such that every component of $T - X$ has less than~$|\G'|$ vertices. It follows from the inductive hypothesis that no component of $T - X$ is zero-forcing for~$\G'$. By Lemma~\ref{separator approach}, $T$ fails to be zero-forcing for~$\G$.
	\end{proof}		
	
%

	One can derive statements similar to Proposition~\ref{small trees} for other classes of graphs by replacing Lemma~\ref{tree sep} with appropriate separator theorems, for example for planar graphs~\cite{liptontarjan}, bounded-genus graphs~\cite{boundedgenus} or graphs excluding a fixed graph as a minor~\cite{alonseymourthomas}. These statements then require more information about~$\G$ than simply its order, however. 
	
%
%
%
	
 		\end{subsection}

\begin{subsection}{Sufficient conditions}

	We now turn to the task of finding sufficient conditions for a graph to be zero-forcing for a finite abelian group~$\G$. Since most graph parameters are bounded on graphs without long paths, Proposition~\ref{paths rule} already provides us with a variety of such conditions. But all of them involve the order of the group and we hope to find estimates based on $D(\G)$ instead, which is often much smaller than~$|\G|$. 
	
	A promising approach is to find a highly connected substructure within the graph to guarantee the existence of connected zero-sum sets under any labeling. As a warm-up, suppose that~$G$ is $D(\G)$-connected and let~$\ell$ be any $\G$-labeling of~$G$. Choose $A \sub V(G)$ maximal with $\ell(A) = 0$. Then $V(G) \setminus A$ necessarily consists of less than $D(\G)$ vertices and~$A$ is non-empty. Since~$G$ is $D(\G)$-connected, $A$ must be connected. Therefore, any condition guaranteeing the existence of a $D(\G)$-connected subgraph (or minor) of~$G$ will ensure that~$G$ is indeed zero-forcing. By a theorem of Mader~\cite{MaderSub}, an average-degree of $4D(\G)$ suffices. Similarly, a famous conjecture of Hadwiger~\cite{Hadwiger} asserts that a graph of chromatic number at least~$k$ always contains the complete graph on~$k$ vertices as a minor. If true, this would imply the following.
	
	\begin{conjecture}
	If $\chi(G) \geq D(\G)$, then~$G$ is zero-forcing for~$\G$.
	\end{conjecture}
	
	Note that if $\chi(G) \geq D(\G)$, then~$G$ contains a subgraph with minimum-degree at least $D(\G) - 1$. Although this does not suffice to force a sufficiently large complete minor, as de~la~Vega observed in~\cite{Vega}, it might just be enough for our purposes.
	
	\begin{conjecture}
	If $\delta(G) \geq D(\G) - 1$, then~$G$ is zero-forcing for~$\G$.
	\end{conjecture}
	
	We introduce a rather abstract notion of a highly connected substructure in a graph. It is strong enough to guarantee a graph to be zero-forcing (Proposition~\ref{sbook} below), but still sufficiently general to be readily implied by other conditions (see Corollary~\ref{extremal function}). Given a graph~$G$ and a positive integer~$k$, call a family~$\K$ of disjoint non-empty connected subgraphs of~$G$ a \emph{scattered bramble of order~$k$} if for every $X \sub V(G)$ with $|X| < k$ there is a unique component of $G - X$ entirely containing some $K \in \K$. By Menger's Theorem, this is equivalent to requiring that for any two $K, K' \in \K$, there is either an edge or a collection of~$k$ internally disjoint paths between them.

		\begin{proposition} \label{sbook}
		Let~$G$ be a graph and~$\K$ a scattered bramble of order~$k$ in~$G$. If $| \K | \geq D(\G)$ and $k+1 \geq D(\G)$, then~$G$ is zero-forcing for~$\G$.
	\end{proposition}
	
	\begin{proof}
	Assume for a contradiction that $(G, \ell)$ was zero-avoiding for some $\G$-labeling~$\ell$. Since $|\K| \geq D(\G)$, we find a non-empty subset $\K' \sub \K$ such that $\sum_{K \in \K'} \ell(K) = 0$. Choose $A \sub V(G)$ connected satisfying $\ell(A \setminus \bigcup \K') = 0$	such that the number of $K \in \K'$ with $K \sub A$ is maximal and, subject to this, $A$ is largest possible. Since any $K \in \K'$ would be a candidate for such a set, we know there is at least one $K_0 \in \K'$ contained in~$A$. We will show that in fact \emph{every} $K \in \K'$ is contained in~$A$, thus showing that $\ell(A) = 0$.
	
	Let $\K^* \sub \K'$ be the set of all $K \in \K'$ with $K \not\sub A$ and assume for a contradiction that $\K^* \neq \emptyset$. Then $K \cap (A \cup N(A)) = \emptyset$ for every $K \in \K^*$: Otherwise, $A \cup K$ would violate the maximality of~$A$. Observe that $N(A)$ separates $K_0 \sub A$ from any $K \in \K^*$. Since~$\K$ is a scattered bramble of order~$k$, this implies that $|N(A)| \geq k \geq D(\G) - 1$. By definition, the sequence
	\[
	\ell(A), \ell(x)_{x \in N(A)}
	\]
	contains a non-empty subsequence summing up to zero. If that subsequence consisted of~$\ell(A)$ and $\ell(x)_{x \in X}$ for some $X \sub N(A)$, we had the connected set $A \cup X$ with $\ell(A \cup X) = 0$, contradicting our assumption that $(G, \ell)$ was zero-avoiding. Therefore there must be a non-empty $B \sub N(A)$ with $\ell(B) = 0$. But then $A \cup B$ violates the maximality of~$A$, since $N(A) \cap \bigcup \K' = \emptyset$.
	\end{proof}

	
	Scattered brambles are a generalization of \emph{$k$-blocks}, where each element of~$\K$ consists of a single vertex (see~\cite{maderblock, kblocks}). But our concept is more versatile: For instance, the rows of a $k \times k$-grid consistute a scattered bramble of order~$k$ and Proposition~\ref{sbook} thus provides a zero-sum theorem for grids.
	
	\begin{corollary} \label{extremal function}
	Let~$G$ be a graph and~$\G$ a finite abelian group with $D(\G) > 2$.
	\begin{enumerate}[(a)]
		\item If $d(G) \geq 2D(\G)-4$, then~$G$ is zero-forcing for~$\G$.
		\item If $\tw(G) \geq \min\{ |\G| - 1, D(\G)^2 - D(\G) - 1 \}$, then~$G$ is zero-forcing for~$\G$.
	\end{enumerate}
	\end{corollary}
	
	\begin{proof}
	(a) By \cite[Theorem~5.3]{kblocks}, $G$ has a minor containing a $(D(\G)-1)$-block of size at least $D(\G)$.
	
	(b) It is easy to see ~-- consider down-closures in a depth-first search tree of~$G$ ~-- that if~$G$ has tree-width at least~$|\G|-1$, then it contains a path on~$|\G|$ vertices and we are done by Proposition~\ref{paths rule}. If~$G$ has tree-width at least $D(\G)^2 - D(\G) - 1$, then by \cite[Lemma~3.2]{reedwood}, $G$ contains a scattered bramble of order~${D(\G) - 1}$ consisting of~$D(\G)$ disjoint paths.
	\end{proof}

\end{subsection}
 		
 		\end{section}
 		
	\begin{section}{Computability and algorithms} \label{logic}
	
		We now address the algorithmic problem of deciding, given a graph~$G$ and a finite abelian group~$\G$, whether~$G$ is zero-forcing for~$\G$. This can be done by testing all possible $\G$-labelings, although that is practically unfeasible even for small graphs and groups.
		
		\begin{question}
			Given a graph~$G$ and (some representation of) a finite abelian group~$\G$, how hard is it computationally to decide whether~$G$ is zero-forcing for~$\G$?
		\end{question}
		
	
	The situation changes when we treat the group~$\G$ as fixed and only the graph~$G$ as input. We know from Corollary~\ref{finite obstruction set} that there is a finite set $Z^{\ind}(\G)$ of graphs such that a graph is zero-forcing for~$\G$ if and only if it contains some graph from $Z^{\ind}(\G)$ as an induced subgraph. There exists, therefore, a very simple polynomial-time algorithm to decide whether a given graph~$G$ is $\G$-zero: Just go through all graphs in $Z^{\ind}(\G)$ and test whether they occur within~$G$.
	
	However, this proof is non-constructive because we have given no indication of how to produce this set $Z^{\ind}(\G)$ in the first place. The following result overcomes this shortcoming partially.
	
	\begin{theorem} \label{computable obstructions}
		There is a computable function $N : \mathbb{N} \to \mathbb{N}$ such that the following holds: If a graph~$G$ is zero-forcing for a finite abelian group~$\G$, but none of its proper induced subgraphs is, then $|G| \leq N(|\G|)$. In particular, the set $Z^{\ind}(\G)$ is computable.
	\end{theorem}
		
	In the proof, we make use of Proposition~\ref{paths rule} and the fact that large graphs without long paths display a certain symmetry (Theorem~\ref{nesetril style 2} below). We then use this symmetry to extend a labeling of an appropriate subgraph to the whole graph. The final ingredient in the proof is the fact that the failure of a labeled graph to be zero-avoiding is always witnessed by a small set of vertices (Lemma~\ref{small certificate}).
	
	The \emph{tree-depth} of a connected graph is defined recursively as follows: Only the graph without any vertices has tree-depth zero and if~$G$ does not have tree-depth at most~$k$ but there is a $v \in V(G)$ such that every component of $G - v$ has tree-depth at most~$k$, then~$G$ has tree-depth $k+1$. When~$G$ is disconnected, its tree-depth is the maximum tree-depth of any of its components. It is easy to see that graphs without paths of length~$k$ have tree-depth at most~$k$.
	
	We first give a straight-forward extension of a theorem of Ne{\v{s}}et{\v{r}}il and Ossona de Mendez~\cite[Theorem~3.1]{nesetril1} that asserts a high degree of symmetry in large graphs of bounded tree-depth.	 Let $(G, \ell)$ be a labeled graph. We call non-empty connected $C_1, \ldots, C_r \sub V(G)$ \emph{clones} if they are pairwise disjoint and for any two $i, j \in [r]$ there is no edge between~$C_i$ and~$C_j$ and there is an $\ell$-preserving automorphism~$\pi$ of~$G$ with $\pi(C_i) = C_j$ that keeps $V(G) \setminus (C_i \cup C_j)$ fixed.

		\begin{theorem} \label{nesetril style 2}
	Let $f : \mathbb{N} \to \mathbb{N}$ non-decreasing. There is a map $F_f : \mathbb{N} \times \mathbb{N} \to \mathbb{N}$ such that the following holds: If~$G$ is a graph of tree-depth less than~$t$, $\ell$ a {$\G$-labeling} of~$G$ with $|\G| \leq N$ and~$G$ has order greater than $F_f(N,t)$, then $(G,\ell)$ contains clones $C_1, \ldots, C_r$ with $r \geq f(|C_1|)$.
	
	If~$f$ is computable, then so is~$F_f$.
	\end{theorem}
	
	\begin{proof}
	We define this function $F := F_f$ by induction on~$t$. For $t=1$, we may take $F(N,1) = 0$ since then the premise is void. For the inductive step, let~$G$ be a graph of tree-depth at most~$t$ and order greater than $F(N,t+1)$, a quantity we will define later. Suppose first that~$G$ is connected. Thus there is a $v \in V(G)$ such that $G - v$ has tree-depth less than~$t$. Define $\ell^*: V(G-v)  \to \G \times \mathbb{Z}_2$ as
	\[
	\ell^*(x) := \begin{cases}
	(\ell(x), 0), &\text{ if } vx \notin E(G) \\
	(\ell(x), 1), &\text{ if } vx \in E(G) .
	\end{cases}
	\]
	As long as we choose $F(N, t+1) > F(2N,t)$, the graph $G - v$ has order greater than $F(2N,t)$ and by inductive hypothesis we find clones $C_1, \ldots , C_r \sub V(G - v)$, $r \geq f(|C_1|)$, in $(G -v, \ell^*)$. By definition of~$\ell^*$, the guaranteed automorphisms of $G - v$ extend to all of~$G$, so that the~$C_j$ are indeed clones in~$(G, \ell)$. 
	
	Assume now that~$G$ was disconnected with components $G_1, \ldots, G_m$. If one of them, say $G_i$, has order greater than $n := F(2N, t) + 1$, then we are done by the previous case, as clones in $(G_i, \ell|_{G_i})$ extend to clones in $(G, \ell)$. Hence we now assume that $|G_i| \leq n$ for all $i \in [m]$. Since there are at most $2^{\binom{n}{2}}$ non-isomorphic connected graphs of order at most~$n$, we find that at least $m/2^{\binom{n}{2}}$ components are isomorphic to the same graph~$H$. There are at most $N^n$ possible $\G$-labelings of~$H$, so $m/(2^{\binom{n}{2}}N^n)$ components of~$G$ carry the same labeling and we have $\ell$-preserving isomorphisms between any two. But then these components are clones in $(G, \ell)$. Since $m \geq |G|/n$, it suffices to take
	\[
	F(N,t+1) := n2^{\binom{n}{2}}N^n f(n). 
	\qedhere
	\]
	
	\end{proof}

	\begin{lemma} \label{small certificate}
		For every finite abelian group~$\G$ there is a constant $s(\G)$ such that the following holds: If $(G, \ell)$ is not zero-avoiding, then there is a non-empty connected $A \sub V(G)$ with $|A| \leq s(\G)$ and $\ell(A) = 0$. 
		
		Moreover, we have $s(\G) \leq |\G|^2$.
	\end{lemma}
	
		Note that this implies the a priori non-obvious fact that, for a fixed finite abelian group~$\G$, we can decide in polynomial time whether a given labeled graph $(G, \ell)$ is zero-avoiding. The proof of this lemma relies on Proposition~\ref{paths rule} and the following simple observation.
	
	\begin{lemma} \label{tree diameter leaves}
		Let~$T$ be a tree of diameter at most~$d$. It~$T$ has at most~$L$ leaves, then
		\[
		|T| \leq \begin{cases}
		dL/2 + 1, &\text{ if $d$ even}\\
		(d-1)L/2 + 2, &\text{ if $d$ odd} .
		\end{cases}
		\]
	\end{lemma}
	
	\begin{proof}
	By induction on~$d$. The assertion is trivial for $d \in \{ 0, 1 \}$. Obtain the tree~$T'$ from~$T$ by deleting all leaves. The diameter of~$T'$ is at most $d-2$ and~$T'$ does not have more leaves than~$T$. The claim follows inductively since $|T| \leq |T'| + L$.
	\end{proof}
	
	\begin{proof}[Proof of Lemma~\ref{small certificate}]
	For the trivial group $\G = \{ 0 \}$ this is clear with $s = 1$. For cyclic groups of prime order, it follows from Theorem~\ref{prime trivial} that we may take $s(\mathbb{Z}_p) = p$. From now on, we assume $|\G| \geq 4$ and prove the statement with $s(\G) := (|\G| - 2)(D(\G) - 1)/2 + 1$.
	
	Let $(G, \ell)$ be a $\G$-labeled graph and assume that $(G, \ell)$ is not zero-avoiding. Choose $A \sub V(G)$ non-empty and connected with $\ell(A) = 0$ such that $|A|$ is minimal. Let~$T$ be a spanning tree of $G[A]$. If~$T$ contains a path~$P$ on~$|\G|$ vertices, then by Proposition~\ref{paths rule} we find a non-empty connected $B \sub P$ with $\ell(B) = 0$. By minimality of~$A$ we must have $|A| = |B| \leq |\G|$ and we are done. If~$T$ had at least $D(\G)$ leaves, then we could find a set~$L$ of leaves with $\ell(L) = 0$, so that $A \setminus L$ would contradict the minimality of~$A$ (except when $|A| = 2 \leq s(\G)$ anyways). Applying Lemma~\ref{tree diameter leaves} with $d := |\G| - 2 \geq 2$ yields $|T| \leq s(\G)$.
	\end{proof}

	\begin{proof}[Proof of Theorem~\ref{computable obstructions}]
	Let $f(n) := |\G|^{n+2}$ and let~$F_f$ be the function guaranteed by Theorem~\ref{nesetril style 2}. We prove the statement with ${N(k) := F_f(1, k+1)}$. Let~$G$ be a graph of order greater than $N(|\G|)$ and assume for a contradiction that $G \in Z^{\ind}(\G)$, that is, $G$ is zero-forcing for~$\G$ but no proper induced subgraph of~$G$ is. By Proposition~\ref{paths rule}, $G$ cannot contain a path on $|\G| + 1$ vertices, so $G$ has tree-depth at most~$|\G|$.
	
	Let $q : V(G) \to \{ 0 \}$ be the constant map. By Theorem~\ref{nesetril style 2}, we find clones $C_1, \ldots, C_r$ with $r \geq f(|C_1|)$ in $(G, q)$. Let $n := |C_1|$ and $H := G - C_r$. By minimality of~$G$, we find a labeling $\ell : V(H) \to \G$ such that $(H, \ell)$ is zero-avoiding. As there are only $|\G| ^n$ possible labelings of the~$C_i$, we find $|\G|^2$ of the~$C_i$ which are labeled identically, say $C_1, \ldots, C_{|\G|^2}$. Extend~$\ell$ to~$G$ by applying the same labeling to~$C_r$. By assumption on~$G$ and Lemma~\ref{small certificate}, there is a connected non-empty $A \sub V(G)$ with $|A| \leq |\G|^2$ and $\ell(A) = 0$. Since $(H, \ell)$ is zero-avoiding, we must have $A \cap C_r \neq \emptyset$. But then there is some $j \in [|\G|^2]$ with $A \cap C_j = \emptyset$. Let~$\pi$ be the automorphism of~$G$ mapping~$C_r$ to~$C_j$ that keeps the rest of $V(G)$ fixed. Then $\pi(A) \sub V(H)$ is connected and by our choice of $\ell|_{C_r}$ we have $\ell( \pi(A)) = \ell(A) = 0$, contrary to our assumption.
	\end{proof}
	
	Our bound on the order of graphs in~$Z^{\ind}(\G)$ is of course useless for all practical purposes. It seems to be an interesting challenge to obtain a better estimate.
	
	\begin{question}
		Is there a polynomial~$P$ such that all graphs in $Z^{\ind}(\G)$ (or $Z(\G)$) have order at most $P( |\G|)$?
	\end{question}
	
	\end{section} 		
	
	\begin{section}{Reconstructing the group} \label{reconstruct}
		
	
	So far we regarded the group~$\G$ as fixed and aimed to understand which properties of a graph made it zero-forcing for~$\G$. But we might as well ask, conversely, what the class of zero-forcing graphs tells us about the group. For instance, we can read off the order of the group as the minimum order of a zero-forcing path (by Proposition~\ref{paths rule}). Can you reconstruct the group from its zero-sum properties on graphs?
	
	\begin{conjecture}
		Let $\G_1, \G_2$ be finite abelian groups. If $Z(\G_1) = Z(\G_2)$, then~$\G_1$ and~$\G_2$ are isomorphic.
	\end{conjecture}
	
		As a possible first step in this direction, we will show that the \emph{exponent} of the group~$\G$ can be recovered by looking at the zero-sum properties of paths and cycles only. Recall that the exponent of~$\G$ is the minimum integer~$m$ such that $mx = 0$ for every $x \in \G$. Equivalently, it is the maximum order of a cyclic subgroup (or quotient) of~$\G$.
		
		\begin{theorem}
		Let~$\G$ be a finite abelian group of exponent~$m$ and let~$C$ be a cycle. Then~$C$ is zero-forcing for~$\G$ if and only if $|C| \geq 1 + \frac{m-1}{m}|\G|$.
	\end{theorem}
	
	\begin{proof}
		Suppose first that $|C| \leq \frac{m-1}{m}| \G|$. Let $\G' \leq \G$ be a subgroup of~$\G$ such that $\G / \G' \cong \mathbb{Z}_m$. Choose $X \sub V(C)$ of size at most $m-1$ so that every component of $C - X$ is a path of order less than $\frac{|\G|}{m} = |\G'|$. By Proposition~\ref{paths rule}, no component of $C - X$ is zero-forcing for~$\G'$ and it follows from Lemma~\ref{separator approach} that~$C$ is not zero-forcing for~$\G$.
		
		Now let $|C| \geq 1 + \frac{m-1}{m}|\G|$ and let the vertices of~$C$ be $v_1, \ldots, v_n$, $n := |C|$, in order of transversal with some arbitrary fixed orientation. If $n \geq |\G|$, then the path on~$|\G|$ vertices is already sufficient for~$C$ to be zero-forcing, so from now on assume $n < |\G|$. For integers $i, j \in [n]$, let $[i,j]$ denote the path from~$v_i$ to~$v_j$ along the fixed orientation. 
		
		Assume for a contradiction that $(C, \ell)$ was zero-avoiding for some ${\ell : V(C) \to \G}$. Define
		\begin{align*}
			a_j &:= \sum_{i=1}^j \ell(v_i) = \ell([1,j]), \; 1 \leq j \leq n \\
			b_j &:= - \sum_{i=j+1}^n \ell(v_i) = - \ell([j+1,n]), \; 1 \leq j \leq n-1.
		\end{align*}
	Since~$\ell$ is zero-avoiding, all of these sums are non-zero, the~$a_j$ are pairwise distinct and so are the~$b_j$. Let $A := \{ a_j \colon j \in [n] \}$ and $B := \{ b_j \colon j \in [n-1] \}$. Then $|A| = n$, $|B| = n-1$ and $A \cup B \sub \G \setminus \{ 0 \}$. Therefore $|A \cap B| \geq 2n - | \G|$. We will show that the intersection of these two sets forces~$C$ to contain a large number of consecutive segments with the same label, ultimately forcing a zero-sum subgraph in this manner.
	
	Suppose that $a_j = b_k$. Then $0 = \ell([1,j]) + \ell([k+1, n])$, so we must have $j \geq k+1$, for otherwise $[1,j] \cup [k+1, n]$ was a zero-sum subgraph of~$C$. Therefore $\ell([1,j]) + \ell([k+1, n]) = \ell(V(C)) + \ell([k+1,j])$ and it follows that $\ell([k+1,j]) = -\ell(V(C))$. 
	
	Construct an auxiliary directed graph~$D$ with vertex-set~$[n]$ where we draw an arc from~$j$ to~$k$ if $a_j = b_k$, which then implies $j \geq k+1$ as we just saw. As the~$b_i$ are pairwise distinct, the outdegree of every vertex is at most one and since the~$a_i$ are pairwise distinct, the indegree of every vertex is at most one. Since~$D$ cannot contain any directed cycles, it follows that $D$ is a union of vertex-disjoint directed paths, say $D = P_1 \cup \ldots P_r$. But then
	\[
	n- r = |E(D)| = |A \cap B| \geq 2n - |\G| .
	\]
	Thus $r \leq |\G| - n$. One of the paths, say~$P_1$, contains at least $(n-r)/r$ edges, so
	\[
	|E(P_1)| \geq \frac{n-r}{r} \geq \frac{n}{|\G| - n} - 1 > m-2,
	\]
	by assumption on~$n$. Let $P_1 = \overrightarrow{j(1)(2)\ldots j(s)}$, for some $s \geq m$. Recall that $j(1) > j(2) > \ldots > j(s)$. Moreover, by construction, $a_{j(i)} = b_{j(i+1)}$ for every $1 \leq i \leq s-1$ and so $\ell( [j(i+1)+1,j(i)] ) = -\ell(V(C))$. Let $x := -\ell(V(C))$ and $S := [j(m)+1,j(1)]$, so that
	\[
	\ell(S) = \sum_{i=1}^{m -1} \ell( [j(i+1)+1,j(i)] ) = (m - 1)x = -x = \ell(V(C)).
	\]
	But $v_1 \notin S$, so $C - S$ is non-empty, connected and satisfies $\ell( C - S) = 0$.	
	\end{proof}
	
	We also expect the follow monotonicty-property for groups.
	
	\begin{conjecture}
		If a graph is zero-forcing for~$\mathbb{Z}_n$, then it is zero-forcing for any finite abelian group of order~$n$.
	\end{conjecture}
	
	\end{section}
 		
	\bibliographystyle{plain}
\bibliography{zsumbib}

 \end{document}